\documentclass[reqno]{amsart}
\usepackage[
    bookmarksopen=true,
    citecolor=blue]{hyperref}
\usepackage{multirow,makecell}
\usepackage{float}
\usepackage{caption}
\usepackage[all]{xy}
\usepackage{eufrak}
\usepackage{mathrsfs, amscd}
\usepackage{amsmath}
\usepackage{amssymb}
\newtheorem{thm}{Theorem}[section]

\newtheorem{lem}{Lemma}[section]
\newtheorem{prop}{Proposition}[section]

\newtheorem{definition}{Definition}[section]
\newtheorem{df}{Definition}[section]

\newtheorem{remark}{Remark}[section]

\newcommand{\nc}{\newcommand}
\nc{\RBA}{{\mathrm{RBA}_\lambda}}
\nc{\sta}{\star_{A}}
\nc{\stb}{\star_{B}}
\nc{\s}{\star}
\nc{\C}{\mathrm{C}}
\nc{\ca}{\C_{\mathrm{Alg}}}
\nc{\crb}{\C_{\mathrm{RBO}_\lambda}}
\nc{\cra}{\C_{\mathrm{RBA}_\lambda}}
\nc{\cm}{\C_{\mathrm{mor}_\lambda}}
\nc{\hm}{\mathrm{H}_{\mathrm{mor}_\lambda}}
\nc{\cdo}{\C_{\mathrm{DO}_\lambda}}
\nc{\cda}{\C_{\mathrm{DA}_\lambda}}

\allowdisplaybreaks[3]

\newcommand{\field}{\mathbb{F}}

\newcommand{\Hom}{{\rm Hom}}

\newcommand{\End}{{\rm End}}

\newcommand{\Id}{\rm Id}

\newcommand{\id}{\rm Id}

\let \t=\otimes
\let \c=\circ

{}

\begin{document}

\title[Cohomologies and deformations of differential algebra morphisms]{Cohomologies and deformations of differential algebra morphisms}



\author[Du]{Lei Du}
\address{School of Mathematical Sciences, Anhui University, Hefei, 230601, China}
\curraddr{}
\email{dulei@ahu.edu.cn}

\author[Bao]{Yanhong Bao}
\address{School of Mathematical Sciences, Anhui University, Hefei, 230601, China}
\curraddr{}
\email{baoyh@ahu.edu.cn}


\thanks{L. Du is the corresponding author. }

\subjclass[2010]{13D10, 16T15}

\keywords{differential algebras; morphisms; cohomologis;  deformations.}

\date{}

\dedicatory{}

\begin{abstract}
This paper studies the formal deformations of differential algebra morphisms. As a consequence, we develop a cohomology theory of differential algebra morphisms to interpret the lower degree cohomology groups as  formal deformations. Then, we prove the Cohomology Comparison Theorem of differential algebra morphisms, i.e., the cohomology of a morphism of differential algebras is isomorphic to the cohomology of an
auxiliary differential algebra. Finally,  we can give a minimal model for morphism of differential algebras with weight=0.
\end{abstract}

\maketitle

\section{Introduction}

Differential algebras originate from the work of Ritt\cite{Ritt} on differential equations, which are developed by Kaplansky\cite{KA}, Kolchin\cite{KO}, Magid\cite{MA}.A commutative differential algebra naturally gives rise to a Novikov algebra by the well-known
work of Gelfand and Dorfman \cite{GE}.Derivations on path algebras and universal enveloping algebras of diﬀerential Lie algebras were studied in \cite{GF,Po}. Recently, deformations of differential algebras are studied in \cite{CGWZ}.

In 1964, Gerstenhaber\cite{GS1} has described  algebraic deformation theory. After that, the cohomology and obstruction theory of various kinds of algebras has been studied widely \cite{BA,lie}. Gerstenhaber and Schack \cite{ms1,ms2} develop a powerful result called the Cohomology Comparison Theorem (CCT) to study a deformation theory of algebra  morphisms.

In this paper, our main objects are differential algebra morphisms.  Precisely, the notion of bimodules  over differential algebra morphisms is defined and cohomologies and deformations of differential algebra morphisms may be established.  In development, the Cohomology Comparison Theorem of differential versions is considered, which says that  the cohomology of a differential algebra morphism is isomorphic to the cohomology of an auxiliary differential algebra morphism.

The paper is organized as follows.  Section 2 reviews some concepts on differential algebras and differential bimodules and provides a cohomology theory of differential algebra morphisms (See Definition \ref{df: cohomology of RB morphism}). Section 3  sets up deformations of differential algebra morphisms, where it is shown that a differential algebra moprhism is rigid if the $2$th-cohomology group is zero (See Theorem \ref{thm: rigid}). Section 4 gives the Cohomology Comparison Theorem (CCT) of  differential algebra morphisms, which says a cohomology of a differential algebra morphism is isomorphic to the cohomology of an auxiliary differential algebra morphism (See Theorem \ref{thm: CCT of RB}). In Section 5, we consider the case of differential algebras with weight=0. Since  differential algebras with weight=0 are Koszul, we can give a minimal model for morphism of differential algebras with weight=0  via the method of Dotsenkol-Poncin\cite{Do}.

In this paper, we work over a field $\field$ of characteristic 0 and unless otherwise specified, linear spaces, linear maps, $\otimes$, $\Hom$, $\End$ are defined over $\field$.

\section{Cohomology of differential algebra morphisms}
In this section, we give the notion of  bimodules over differential algebra morphisms and  establish  cohomologies of differential algebra morphisms with coefficients in their modules. First we recall some concepts and results on differential algebras and differential bimodules.

\begin{df}\label{Def:differential algebra} {\rm (\cite{LO})}
Let $A$ be an associative algebra over field $\field$ and
$\lambda\in \field$. A linear
operator $d: A\rightarrow A$ is said to be a differential operator of
weight $\lambda$ if it satisfies
\begin{eqnarray}\label{Eq: Rota-Baxter relation}
d(u\cdot v)=d(u)v+ud(v)+\lambda d(u)d(v)\end{eqnarray}	
for any $u,v \in A$. In this case,  $(A, d)$ is called a differential
algebra of weight $\lambda$.
\end{df}

\begin{df}
Let $(A, d_A)$ and $(B,d_B)$ be differential algebras. A linear map $\phi:A\rightarrow B$ is called a morphism of differential algebras from $(A, d_A)$ to $(B,d_B)$
if $\phi$ is an algebra morphism and satisfies $\phi\circ d_A= d_B \circ \phi$.
\end{df}
Denote by $\mathbf{Diff}$ the category of differential algebras of weight $\lambda$ with above morphisms.
\begin{df}\label{Def: Rota-Baxter bimodules}
Let $(A, d_A)$ be a differential algebra of weight $\lambda$.  A pair $(M, d_M)$ is called a bimodule
over differential algebra $(A, d_A)$ if $M$ is a bimodule
over associative algebra $A$ and $d_M: M\rightarrow M$ satisfies
\begin{align}
d_M(ax)&=d_A(a)x+ad_M(x)+\lambda d_A(a)d_M(x),\\
d_M(xa)&=xd_A(a)+d_M(x)a+\lambda d_M(x)d_A(a),
\end{align}
 for any $a\in A$ and $x\in M$.
\end{df}
Of course, $(A, d_A)$ is  a differential bimodule over itself.

\begin{df}
Let $(M, d_M)$ and $(N, d_N)$ be bimodules of differential algebra $(A,d_A)$. A linear map  $f : M\rightarrow N$ is called a morphism of $(A,d_A)$-bimodule if $f$ is a $A$-bimodule morphism and such that $d_N\circ f=f\circ d_M$.
\end{df}

Given a differential bimodule $M = (M, d_M)$ over the differential algebra $A = (A, d_A)$, a new
differential bimodule structure on $M$, with the same derivation $d_M$, is given by
\begin{align}
a\triangleright m=(a+\lambda d_A(a))m, m\triangleleft a=m(a+\lambda d_A(a)),m\in M, a\in A.
\end{align}
For distinction, we let $_\triangleright M_\triangleleft$ denote this new differential bimodule structure over $(A, d_A)$;
for details, see \cite{GY}.

In \cite{ms1}, Gerstenhaber and Schack define the bimodule of an associative algebra morphism, that is, let $\phi:A\rightarrow B$ be a morphism of associative algebras , then a $\phi$-bimodule is a triple $\langle M,N ,\psi\rangle$ such
that $M$ is a  bimodule over $A$, $N$ is a  bimodule over $B$, and $\psi: M,\rightarrow N$
is an $A$-bimodule morphism, where $N$ is considered as a bimodule over associative algebra $A$ in a natural way.\\
\indent Similarly, let $\phi:(A, d_{A})\rightarrow (B, d_{B})$ be a morphism of differential algebras of weight $\lambda$, then a differential $\phi$-bimodule is a triple $\langle M,N ,\psi\rangle$ such
that $(M,d_M)$ is a differential bimodule over $(A, d_{A})$, $(N,d_N)$ is a differential bimodule over $(B, d_{B})$, and $\psi: (M,d_M)\rightarrow (N,d_N)$
is a $(A, d_{A})$-bimodule morphism, where $(N,d_N)$ is regarded as a bimodule over differential algebra $(A, d_{A})$ in a natural way.

\begin{thm}\label{thm:triang action Module-morphism}
Let $\psi:(M,d_M)\rightarrow(N,d_N)$ be a morphism of $(A, d_{A})$-bimodules, then $\psi:({}_{\rhd}M_\lhd,d_M) \rightarrow({}_\rhd N_\lhd,d_N) $ is a morphism of $(A, d_{A})$-bimodules.
\end{thm}

\begin{proof}
Let $_{\rhd}M_\lhd$ and $_\rhd N_\lhd$ be $(A, d_{A})$-bimodules. It suffice to show that $\psi:({}_{\rhd}M_\lhd,d_M) \rightarrow({}_\rhd N_\lhd,d_N) $ is a morphism of $(A_{\s}, d_{A})$-bimodules.

Since we have $d_{N} \circ \psi= \psi \circ d_M$, it follows that

\begin{align*}
\psi(a\rhd m)&=\psi((a+\lambda d_A(a))x)\\
          &=\psi(ax)+\lambda \psi(d_A(a)x)\\
          &=a\psi(x)+\lambda d_A(a)\psi(x)\\
          &=(a+\lambda d_A(a))\psi (x)\\
          &=a\rhd \psi(x).
\end{align*}
Similarly, $\psi(x\lhd a)=\psi(x)\lhd a$. Therefore, $\psi$ is a morphism of $(A, d_{A})$-bimodules.
\end{proof}

\begin{remark}\label{remark morphism module}
Let $\phi:(A, d_{A})\rightarrow (B, d_{B})$ be a morphism of differential algebras and $\langle M,N ,\psi\rangle$ be a differential $\phi$-bimodule. By Theorem \ref{thm:triang action Module-morphism},
$\psi:({}_{\rhd}M_\lhd,d_M)\rightarrow({}_{\rhd}N_\lhd,d_N)$ is a morphism
 of $(A, d_{A})$-bimodules, that is,
$\langle{}_{\rhd}M_\lhd,{}_\rhd N_\lhd ,\psi\rangle$ is still a differential $\phi$-bimodule.
To avoid confusion in the following context, we denote it by
$\langle{}_{\rhd}M_\lhd,{}_\rhd N_\lhd,_\triangleright\psi_\triangleleft\rangle$.
\end{remark}

Let $A$ be an associative algebra, $M$ be a bimodule over $A$. The Hochschild cochain complex of $A$ with coefficients in $M$ is $$\C^\bullet_{\mathrm{Alg}}(A,M):=\bigoplus\limits_{n=0}^\infty \C^n_{\mathrm{Alg}}(A,M),$$
 where $\C^n_{\mathrm{Alg}}(A,M)=\Hom(A^{\t n},M)$ and the differential $\partial_{\rm Alg}^n: \C^n_{\mathrm{Alg}}(A,M)\rightarrow \C^{n+1}_{\mathrm{Alg}}(A,M)$ is defined as:
 \begin{align*}
\partial_{\rm Alg}^n(f)(a_{1, n+1} )= &  a_1f(a_{2, n+1})+\sum\limits_{i=1}^n(-1)^{i}f(a_{1, i-1}\t a_i a_{i+1}\t  a_{i+2, n+1})\\
	 &+ (-1)^{n+1}f(a_{1, n})a_{n+1}
\end{align*}
for all $f\in \C^n_{\mathrm{Alg}}(A,M), a_{1,n}=a_1,\dots,a_{n+1}\in A$.

 The cohomology of the Hochschild cochain complex $\C^\bullet_{\mathrm{Alg}}(A,M)$ is called the Hochschild cohomology of $A$ with coefficients in $M$,  denoted by $\mathrm{H}^\bullet(A,M)$.

\begin{df}{\rm (\cite{CGWZ})}
 	Let $A=(A,d_A)$ be a differential algebra of weight $\lambda$ and $(M,d_M)$ be a differential bimodule over it. Then the cochain complex $(\ca^\bullet(A, {_\rhd M_\lhd}),\partial_{\rm Alg}^\bullet)$ is called the cochain complex of differential operator $d_A$ with coefficients in $(M, d_M)$,  denoted it by $(\cdo^\bullet(A, M),\partial^\bullet_{{\rm DO}_\lambda})$.
 \end{df}

\begin{df}\label{def:map_cone_da_to_do}{\rm (\cite{CGWZ})}
Let $(A,d_A)$ be a differential algebra of weight $\lambda$, $(M,d_M)$ be a  differential bimodule over it. Define a chain map   $\Phi^\bullet:\ca^\bullet(A,M) \rightarrow \cdo^\bullet(A,M)$ as following
\begin{enumerate}
\item[{\rm (1)}]$\Phi^0(x)=-d_M(x), x\in C^0_{Alg}(A,M)=M$,

\item[{\rm (2)}]When  $n\geqslant 1$ and $ f\in \ca^n(A,M)$,  define $\Phi^n(f)\in \cdo^n(A,M)$ by:
\begin{align*}
 &\Phi^n(f)(a_{1,n}) \\
=&\sum_{k=1}^n\lambda^{k-1}\sum_{1\leqslant i_1<\cdot\cdot\cdot<i_k\leqslant n}f(a_{1,i_1-1}\otimes d_A(a_{i_1})\otimes a_{i_1+1,i_2-1}\otimes d_A(a_{i_2})\otimes\cdot\cdot\cdot d_A(a_{i_k})\otimes a_{i_{k+1},n})\\
&-d_M(f(a_{1,n})).
\end{align*}
\end{enumerate}

Then  the  cochain complex $(\cda^\bullet(A,M), \partial_{{\rm DA}_\lambda}^\bullet)$  of differential algebra $(A,d_A)$ with coefficients in $(M,d_M)$ is defined by the negative shift of the mapping cone of $\Phi^\bullet$. Exactly,
\[\cda^0(A,M)=\ca^0(A,M),   \cda^n(A,M)=\ca^n(A,M)\oplus \cdo^{n-1}(A,M), \forall n\geqslant 1,\]
 and the differential $\partial_{{\rm DA}_\lambda}^n: \cda^n(A,M)\rightarrow \cda^{n+1}(A,M)$ is given by \[\partial_{{\rm DA}_\lambda}^n(f,g)= (\partial_{\rm Alg}^n(f), -\partial_{{\rm DO}_\lambda}^{n-1}(g)  -\Phi^n(f))\]
 for any $f\in \ca^n(A,M)$ and $g\in \cdo^{n-1}(A,M)$.

Futher, the  cohomology of $(\cda^\bullet(A,M), \partial_{{\rm DA}_\lambda}^\bullet)$ is denoted by $\mathrm{H}_{\mathrm{DA}_\lambda}^\bullet(A,M)$.
\end{df}

In \cite{ms1}, Gerstenhaber and Schack have studied a cohomology thoery of associative algebra morphisms. Let $\phi: A\rightarrow B$ be a morphism of associative algebras and $\langle M,N,\psi\rangle$ be a $\phi$-bimodule. The cochain complex of $\phi$ with coefficients in $\langle M,N,\psi\rangle$ is $\C^\bullet (\phi,\psi)$, where $\C^n (\phi,\psi)=0$, for $n< 0$; $\C^0 (\phi,\psi)=\C^0_{\mathrm{Alg}} (A,M)\oplus\C^0_{\mathrm{Alg}} (B,N)$ and
for $n\geq1, \C^n (\phi,\psi)=\C^n_{\mathrm{Alg}} (A,M)\oplus\C^{n}_{\mathrm{Alg}} (B,N)\oplus \C^{n-1}_{\mathrm{Alg}} (A,N)$.
The differential $\delta^n:\ca^n(\phi,\psi)\rightarrow \ca^{n+1}(\phi,\psi)$ is defined by
$$\delta^n(f,g,h)=(\partial_{\rm Alg}^n f,\partial_{\rm Alg}^n g,\psi \circ f-g\circ \phi^{\t n}-\partial_{\rm Alg}^{n-1}h).$$

Let $\phi:(A, d_{A})\rightarrow (B, d_{B})$ be a morphism of differential algebras of weight $\lambda$ and $\langle M,N ,\psi\rangle$ be a differential $\phi$-bimodule. Then, $\langle{}_{\rhd}M_\lhd,{}_\rhd N_\lhd ,_\triangleright\psi_\triangleleft\rangle$ is still a differential $\phi$-bimodule.  Hence, we may construct a cochain complex that controls  deformations of  differential algebra morphisms as follows.

\begin{prop}\label{prop:chain map}
Suppose that  $\C^\bullet(\phi,\psi)$ is a Hochschild cochain complex of $\phi$ with coefficients in $\psi$ and $\C^\bullet(\phi,_\triangleright\psi_\triangleleft)$ is a Hochschild cochain complex of $\phi$ with coefficients in $_\triangleright\psi_\triangleleft$. $\pi^\bullet: \C^\bullet(\phi,\psi)\rightarrow \C^\bullet(\phi,_\triangleright\psi_\triangleleft)$ is defined as
\begin{enumerate}
\item[{\rm (1)}] $\pi^0: \C^0(\phi,\psi)\rightarrow \C^0(\phi,_\triangleright\psi_\triangleleft)$ to
 be identity map,
\item[{\rm (2)}] when $n\geqslant 1$, $\pi^n: \C^n(\phi,\psi)\rightarrow \C^n(\phi,_\triangleright\psi_\triangleleft)$ is defined as
 $$\pi^n(f,g,h)=(\Phi_{A,M}^n(f),\Phi_{B,N}^n(g),\Phi_{A,N}^{n-1}(h))$$
 for $f\in\ca^n(A,M)$, $g\in\ca^n(B,N)$ and $h\in\ca^{n-1}(A,N)$.
 \end{enumerate}
 Then $\pi^\bullet: \C^\bullet(\phi,\psi)\rightarrow \C^\bullet(\phi,_\triangleright\psi_\triangleleft)$ is a chain map.
\end{prop}

\begin{proof}
Only need to prove any $(f,g,h)\in \C^n(\phi,\psi)$,
$\pi^n\delta^n(f,g,h)=\delta^n\pi^n(f,g,h)$, i.e.,
\begin{align*}
&(\Phi_{A,M}^{n+1}(\partial_{\rm Alg}^n(f)),\Phi_{B,N}^{n+1}(\partial_{\rm Alg}^n(g)), \Phi_{A,N}^{n}(\psi\circ f)-\Phi_{A,N}^{n}( g\circ \phi^{\t n})-\Phi_{A,N}^{n}(\partial_{\rm Alg}^{n-1}h))\\
=&(\partial_{{\rm DO_\lambda}}^n(\Phi_{A,M}^{n}(f)),\partial_{{\rm DO_\lambda}}^n(\Phi_{B,N}^{n}(g)),\psi\circ(\Phi_{A,M}^{n}(f))-(\Phi_{B,N}^{n}(g))\circ \phi^{\t n}-\partial_{{\rm DO_\lambda}}^{n-1}(\Phi_{A,N}^{n-1}(h))).
\end{align*}
Since $\Phi^{\bullet}$ is a chain map, we only need to show  $\Phi_{A,M}^{n}(\psi\circ f)=\psi\circ(\Phi_{A,M}^{n}(f))$ and $\Phi_{B,N}^{n}( g\circ \phi^{\t n})=(\Phi_{B,N}^{n}(g))\circ \phi^{\t n}$. It is obvious due to $d_N\circ \psi=\psi \circ d_M$ and $d_B\circ \phi=\phi \circ d_A$.
\end{proof}
At the end of this section, we define cohomologies of  differential algebra morphisms with coefficients in their modules.
\begin{df}\label{df: cohomology of RB morphism}
Let $\pi^\bullet: \C^\bullet(\phi,\psi)\rightarrow \C^\bullet(\phi,_\triangleright\psi_\triangleleft)$ be a chain map defined in Proposition {\rm \ref{prop:chain map}}. We may define a cochain complex $\cm^\bullet(\phi,\psi)$ to be the negative shift of the mapping cone of $\pi^{\bullet}$, that is,
\begin{enumerate}
\item[(1)] $\cm^0(\phi,\psi)=\C^0(\phi,\psi)$,
\item[(2)]when $n\geqslant 1$, $\cm^n(\phi,\psi)=\C^n(\phi,\psi)\oplus \C^{n-1}(\phi,_\triangleright\psi_\triangleleft)$,
\end{enumerate}
and the differential map $\varrho^n :\cm^n(\phi,\psi)\rightarrow \cm^{n+1}(\phi,\psi)$ is defined by:

when $n=1$,
\begin{align}
&\varrho^1\big((f,g,n),(m_1,n_1)\big)\\\nonumber
=&\big(\delta^1(f,g,n),\delta^0(m_1,n_1)-\pi^1(f,g,n_1)\big)\\\nonumber
=&\big((\partial_{\rm Alg}^1(f_1),\partial_{\rm Alg}^1(g_1),\psi\circ f-g\circ\phi),-\Phi_{A,M}^1 (f),-\Phi_{B,N}^1 (g) , \psi(m_1)-n_1-
\Phi_{A,N}^{0}(h_1) \big),
\end{align}
when $n\geqslant 2$,
\begin{align}
&\varrho^n\big((f_1,g_1,h_1),(f_2,g_2,h_2)\big)\\\nonumber
=&\big(\delta^n(f_1,g_1,h_1),\delta^{n-1}(f_2,g_2,h_2)+(-1)^n\pi^n(f_1,g_1,h_1)\big)\\\nonumber
=&\big((\partial_{\rm Alg}^n(f_1),\partial_{\rm Alg}^n(g_1),\psi\circ f_1-g_1\circ\phi^{\t n}-\partial_{\rm Alg}^{n-1}(h_1)),\\\nonumber
&(\partial_{\rm DO_\lambda}^{n-1}(f_2)+(-1)^n\Phi_{A,M}^n (f_1) , \partial_{\rm DO_\lambda}^{n-1}(g_2)+(-1)^n\Phi_{B,N}^n (g_1) , \\\nonumber
&\psi\circ f_2-g_2\circ\phi^{\t n-1}-\partial_{\rm DO_\lambda}^{n-2}(h_2)+(-1)^{n-1}
\Phi_{A,N}^{n-1}(h_1) \big),
\end{align}
for $(f_1,g_1,h_1)\in \C^n(\phi,\psi), (f_2,g_2,h_2)\in \C^{n-1}(\phi,_\triangleright\psi_\triangleleft)$ and $\Phi^\bullet$ is defined by {\rm Definition} {\rm \ref{def:map_cone_da_to_do}}.

Denote the  cohomology group of $(\cm^\bullet(\phi,\psi), \varrho^\bullet)$ by $\mathrm{H}_{\mathrm{mor}_\lambda}^\bullet(\phi,\psi)$.
\end{df}

\section{Deformations of   differential morphisms}
In this section, we study formal deformations of a  differential algebra morphisms and consider the
rigidity of differential algebra morphisms.

Let $(A,\mu_A,d_A)$ and $(B,\mu_B,d_B)$ be  differential algebras of weight $\lambda$ and $\phi:(A, d_{A})\rightarrow (B, d_{B})$ be a morphism of differential algebras. For $X=\{A,B\}$, define
\begin{align*}
\mu_{X,t}&=\sum_{i=0}^{\infty}\mu_{X,i} t^i,~~~\mu_{X,0}=\mu_{X},\\
d_{X,t}&=\sum_{i=0}^{\infty}d_{X,i} t^i,~~~d_{X,0}=d_{X},\\
\phi_{t}&=\sum_{i=0}^{\infty}\phi_{i} t^i,~~~\phi_{0}=\phi.
\end{align*}

 $(\mu_{A,t},d_{A,t},\mu_{B,t},d_{B,t},\phi_{t})$ is called a $1$-parameter formal deformation of $\phi:(A, d_{A})\rightarrow (B, d_{B})$, if $(A[[t]],\mu_{A,t},d_{A,t})$, $(B[[t]],\mu_{B,t},d_{B,t})$ are $\field [[t]]$-differential algebras of weight $\lambda$ and $\phi_{t}: A[[t]]\rightarrow B[[t]]$ is a morphism of differential algebras.

Power series $(\mu_{A,t},d_{A,t},\mu_{B,t},d_{B,t},\phi_{t})$ is a 1-parameter formal deformation of $\phi:(A, d_{A})\rightarrow (B, d_{B})$ if and only if for any $a_1,a_2,a_3\in A$, $b_1,b_2,b_3\in B$ the following equations hold :
\begin{align*}
 \mu_{A,t}(a_1\t \mu_{A,t}(a_2\t a_3))&=\mu_{A,t}(\mu_{A,t}(a_1\t a_2)\t a_3),\\
 d_{A,t}\mu_{A,t}(a_1\t a_2)&=\mu_{A,t}(d_{A,t}(a_1)\t a_2)+\mu_{A,t}(a_1\t d_{A,t}(a_2))+\lambda u_{A,t}(d_{A,t}(a_1)\t d_{A,t}(a_2))\Big),\\
 \mu_{B,t}(b_1\t \mu_{B,t}(b_2\t b_3))&=\mu_{B,t}(\mu_{B,t}(b_1\t b_2)\t b_3),\\
 d_{B,t}\mu_{B,t}(b_1\t b_2)&=\mu_{B,t}(d_{B,t}(b_1)\t b_2)+\mu_{B,t}(b_1\t d_{B,t}(b_2))+\lambda u_{B,t}(d_{B,t}(b_1)\t d_{B,t}(b_2))\Big),\\
 \phi_{t}(\mu_{A,t}(a_1\t a_2))&=\mu_{B,t}(\phi_{t}(a_1)\t\phi_{t}(a_2)),\\
 d_{B,t}\phi_{t}(a_1)&=\phi_{t}d_{A,t}(a_1).
 \end{align*}

By expanding these equations and comparing the coefficients of $t^n$,   the following equations hold : for any $n\geqslant 0$,
\begin{equation}\label{Eq: deform eq for  products in DA}
	\sum_{i=0}^n\mu_{X,i}\circ(\mu_{X,n-i}\t \id)=\sum_{i=0}^n\mu_{X,i}\circ(\id\t \mu_{X,n-i}),\end{equation}
\begin{align}\label{Eq: Deform RB operator in DA}
&\sum_{i+j=n\atop i, j\geqslant 0}	d_{X,i}\circ\mu_{X,j} \nonumber\\
=&\sum_{i+j=n\atop i, j\geqslant 0} \mu_{X,i}\circ(\id\t d_{X,j})
  +\sum_{i+j=n\atop i, j\geqslant 0} \mu_{X,i}\circ(d_{X,j}\t \id),
\end{align}
\begin{equation}\label{Eq: deform eq for  morphism of multipication}
\sum_{i=0}^n\phi_i\circ\mu_{A,j}=\sum_{i+j+k=n\atop i, j, k\geqslant 0}\mu_{B,i}
\circ (\phi_j\t\phi_k),
	\end{equation}
\begin{equation}\label{Eq: deform eq for  morphism of operators}
	\sum_{i=0}^nd_{B,i}\circ\phi_{n-i}=\sum_{i=0}^n\phi_{i}\circ d_{A,n-i},\end{equation}
where $X=\{A,B\}$.

\begin{prop}
Let $(\mu_{A,t},d_{A,t},\mu_{B,t},d_{B,t},\phi_{t})$ be a 1-parameter formal deformation of $\phi:(A, d_{A})\rightarrow (B, d_{B})$, then $\big((\mu_{A,1},\mu_{B,1},\phi_{1}),(d_{A,1},d_{B,1},0)\big)$ is a 2-cocycle in the cochain complex $\cm^\bullet(\phi,\phi)$.
\end{prop}

\begin{proof}
We only need to show that
\begin{align*}
&\varrho^2\big((\mu_{A,1},\mu_{B,1},\phi_{1}),(d_{A,1},d_{B,1},0)\big)\\
=&\big((\partial_{\rm Alg}^2(\mu_{A,1}),\partial_{\rm Alg}^2(\mu_{B,1}),\phi\circ \mu_{A,1}-\mu_{B,1}\circ\phi^{\t 2}-\partial_{\rm Alg}^{1}(\phi_{1})),\\\nonumber
&(-\partial_{\rm Alg}^{1}(d_{A,1})-\Phi_{A}^2 (\mu_{A,1}) , -\partial_{\rm Alg}^{1}(d_{B,1})-\Phi_B^2 (\mu_{B,1}) , -\phi\circ d_{A,1}+d_{B,1}\circ\phi-
\Phi_{A,B}^{1}(\phi_{1}) \big)\\
=&0.
\end{align*}
Due to  deformations of algebra morphism \cite{MS1} and deformations of differential algebras \cite{CGWZ}, it suffices to show $-\phi\circ d_{A,1}+d_{B,1}\circ\phi-
\Phi_{A,B}^{1}(\phi_{1})=0$.

When $n=1$,   Equations~(\ref{Eq: deform eq for  morphism of operators}) becomes
$$d_{B}\circ\phi_{1}+d_{B,1}\circ\phi=\phi_{1}\circ d_{A}+\phi\circ d_{A,1}$$
that is, $-\phi\circ d_{A,1}+d_{B,1}\circ\phi-
\Phi_{A,B}^{1}(\phi_{1})=0$, so  $\big((\mu_{A,1},\mu_{B,1},\phi_{1}),(d_{A,1},d_{B,1},0)\big)$ is a 2-cocycle.
\end{proof}

\begin{df}
The 2-cocycle $\big((\mu_{A,1},\mu_{B,1},\phi_{1}),(d_{A,1},d_{B,1},0)\big)$ is called an infinitesimal deformation of the $1$-parameter formal deformation $(\mu_{A,t},d_{A,t},\mu_{B,t},d_{B,t},\phi_{t})$.
\end{df}

Let $(\mu_{A,t},d_{A,t},\mu_{B,t},d_{B,t},\phi_{t})$ and $(\mu'_{A,t},d'_{A,t},\mu'_{B,t},d'_{B,t},\phi'_{t})$ be two 1-parameter formal deformations of $\phi:(A, d_{A})\rightarrow (B, d_{B})$. Then a formal isomorphism from $(\mu_{A,t},d_{A,t},\mu_{B,t},d_{B,t},\phi_{t})$ to
$(\mu'_{A,t},d'_{A,t},\mu'_{B,t},d'_{B,t},\phi'_{t})$  are two power series $(F_{A,t},F_{B,t})$ with
\begin{align*}
F_{A,t}&=\sum_{i=0}^{\infty}F_{A,i}t^i:A[[t]]\rightarrow A[[t]],~~F_{A,i}
\in \Hom(A,A), F_{A,0}=\id_A, \\
F_{B,t}&=\sum_{i=0}^{\infty}F_{B,i}t^i:B[[t]]\rightarrow B[[t]],~~F_{B,i}
\in \Hom(B,B), F_{B,0}=\id_B,
\end{align*}
such that for $X=\{A,B\}$
\begin{align}
\label{Eq: equivalent deformations 1}
F_{X,t}\circ\mu_{X,t}&=\mu'_{X,t}\circ(F_{X,t}\t F_{X,t}),\\
\label{Eq: equivalent deformations 2}
F_{X,t}\circ d_{X,t}&=d'_{X,t}\circ F_{X,t},\\
\label{Eq: equivalent deformations 3}
F_{B,t}\circ\phi_t &=\phi'_t\circ F_{A,t} .
\end{align}

\begin{prop}
The infinitesimal deformations of two equivalent $1$-parameter formal deformations of $\phi:(A,d_A)\rightarrow (B,d_B)$ are in the same cohomology class in $\hm^\bullet$.
\end{prop}
\begin{proof}
Let $(F_{A,t},F_{B,t})$ be a formal isomorphism from $(\mu_{A,t},d_{A,t},\mu_{B,t},d_{B,t},\phi_{t})$ to
$(\mu'_{A,t},d'_{A,t},\mu'_{B,t},d'_{B,t},\phi'_{t})$.  We have for $X=\{A,B\}$
\begin{align*}
&\mu_{X,1}-\mu'_{X,1}=\mu_X\circ(\id\t F_{X,1})-F_{X,1}\circ\mu_{X}+\mu_{X}\circ(F_{X,1}\t \id),\\
&d_{X,1}-d'_{X,1}=d_{X}\circ F_{X,1}-F_{X,1}\circ d_{X},\\
&\phi_{1}-\phi'_{1}=\phi\circ F_{A,1}-	F_{B,1}\circ \phi.
\end{align*}
Thus,
$$\big((\mu_{A,1},\mu_{B,1},\phi_{1}),(d_{A,1},d_{B,1},0)\big)-\big((\mu'_{A,1},\mu'_{B,1},\phi'_{1}),(d'_{A,1},d'_{B,1},0)\big)
=\varrho^1(F_{A,1},F_{B,1}).$$
\end{proof}

\begin{df}
	A differential algebra moprhism $\phi:(A,d_A)\rightarrow (B,d_B)$ is called to be rigid if every $1$-parameter formal deformation is equivalent to $(\mu_{A},d_{A},\mu_{B},d_{B},\phi)$.
\end{df}

\begin{thm}\label{thm: rigid}
Let $\phi:(A,d_A)\rightarrow (B,d_B)$ be a differentail algebra moprhism. If $\hm^2(\phi,\phi)=0$, then $\phi:(A,d_A)\rightarrow (B,d_B)$ is rigid.
\end{thm}
\begin{proof}
Let $(\mu_{A,t},d_{A,t},\mu_{B,t},d_{B,t},\phi_{t})$ be a 1-parameter formal deformation of $\phi:(A,d_A)\rightarrow (B,d_B)$, then $\big((\mu_{A,1},\mu_{B,1},\phi_{1}),(d_{A,1},d_{B,1},0)\big)$ is a 2-cocycle. By $\hm^2(\phi,\phi)=0$, there exists $\big((F'_1,G'_1,b),(a_1,b_1)\big)\in
\cm^1(\phi,\phi)$ such that $\varrho\big((F'_1,G'_1,b),(a_1,b_1)\big)=\big((\mu_{A,1},\mu_{B,1},\phi_{1}),
(d_{A,1},d_{B,1},0)\big)$, that is,
\begin{align*}
\mu_{A,1}(x\t y)&=xF'_1(y)-F'_1(xy)+F'_1(x)y  \indent\text{for} ~x,y\in A\\
\mu_{B,1}(x\t y)&=xG'_1(y)-G'_1(xy)+G'_1(x)y   \indent\text{for} ~x,y\in B\\
\phi_1(x)&=\phi F'_1(x)-G'_1\phi(x) \indent\text{for}~ x\in A\\
d_{A,1}(x)&=-\Phi^1(F'_1(x))   \indent\text{for} ~x\in A\\
d_{B,1}(y)&=-\Phi^1(G'_1(y))  \indent\text{for} ~y\in B\\
d_B(b)&=b_1-\phi(a_1).
\end{align*}

Set $F_{A,t}=\Id_A+F_{A,1}t$, $F_{B,t}=\Id_B+F_{B,1}t$, and define $(\mu'_{A,t},d'_{A,t},\mu'_{B,t},d'_{B,t},\phi'_{t})$ by
\begin{align}
\begin{split}\label{eq:Thm le}
\mu'_{A,t}&=F_{t}\circ\mu_{A,t}\circ(F^{-1}_{t}\t F^{-1}_{t}),\\
d'_{A,t}&=F_{t}\circ d_{A,t}\circ F^{-1}_{t},\\
\mu'_{B,t}&=G_{t}\circ\mu_{B,t}\circ(G^{-1}_{t}\t G^{-1}_{t}),\\
d'_{B,t}&=G_{t}\circ d_{B,t}\circ G^{-1}_{t},\\
\phi'_t&=G_{t}\circ\phi_t\circ F^{-1}_{t}.
\end{split}
\end{align}
Let $X=\{A,B\}$, we get
\begin{align*}
\mu'_{X,t}&=\mu_X+[F_{X,1}\mu_X+\mu_{A,1}-\mu_A({\Id}\otimes F_{X,1}+F_{X,1}\otimes {\Id})]t+\mu'_{X,2}t^2+\cdots,\\
&=\mu_X+\mu'_{X,2}t^2+\cdots\\
d'_{X,t}&=d_{X}+(F_{X,1}d_X+d_{X,1}-d_XF_{X,1})t+d'_{X,2}t^2+\cdots,\\
&=d_{X}+d'_{X,2}t^2+\cdots\\
\phi'_t&=\phi+(F_{B,1}\c\phi+\phi_1-\phi\c F_{A,1})t+\phi'_{2}t^2+\cdots\\
&=\phi+\phi'_{2}t^2+\cdots.
\end{align*}
Futher, we may verify that $(\mu'_{A,2},\mu'_{B,2},\phi'_{2},d'_{A,2},d'_{B,2},0)$ is also a $2$-cocyle. Then, by repeating the argument,  it is equivalent to a trivial deformation. Thus, $\phi:(A,d_A)\rightarrow (B,d_B)$ is rigid.
\end{proof}




\section{CCT theorem of differentail algebra morphisms}
In this section, we may study the Cohomologies Comparison Theorem (CCT) of  differentail algebra morphisms, which may say that the cohomology of a differentail algebra morphism is isomorphic to the cohomology of an auxiliary differentail algebra.

\begin{df}{\rm (\cite{ms1})}\label{df: morphism algebra}
Let $A,B$ be two associative algebras, $\phi: A\rightarrow B$ be a morphism of associative algebras, the mapping ring $\phi !$ is defined as $\phi !=A\oplus B\oplus B\phi$, the multipication is determined by associativity, the products in $A,B$ and the conditions: $a\cdot b=b\cdot a=\phi \cdot b=a\cdot \phi=\phi^2=0$ and $\phi \cdot a=\phi(a)\phi$, the unit of $\phi !$ is $1_A+1_B$. In other words, the elements of $\phi !$ are of the form $a+b_1+b_2\phi$ with $a\in A;b_1,b_2\in B$ and $(a+b_1+b_2\phi)(a'+b'_1+b'_2\phi)=aa'+b_1b'_1+(b_2\phi(a')+b_1b'_2)\phi$.
\end{df}

\begin{df}{\rm (\cite{ms1})}
Let $\psi:M\rightarrow N$ be a bimodule of $\phi: A\rightarrow B$, the $\phi !$-module $\psi !$ is defined by $\psi !=M\oplus N\oplus N \phi$, for $m+n_1+n_2\phi \in\psi !$, $a+b_1+b_2\phi\in\phi !$, the bimodule structure of $\phi !$ on
$\psi !$ is given by
\begin{align*}
(a+b_1+b_2\phi)(m+n_1+n_2\phi)=am+b_1n_1+(b_2\psi(m)+b_1n_2)\phi,\\
(m+n_1+n_2\phi)(a+b_1+b_2\phi)=ma+n_1b_1+(n_2\phi(a)+n_1b_2)\phi.
\end{align*}
\end{df}

\begin{lem}\label{lem: phi RB algebra}
Let $\phi:(A,d_A)\rightarrow (B,d_B)$ be a morphism of differentail algebras.  Define a linear map  $d_{\phi !}:\phi !\rightarrow \phi !$ by $d_{\phi !}(a+b_1+b_2\phi)=d_{A}(a)+d_{B}(b_1)+d_{B}(b_2)\phi$, where $\phi !=A\oplus B\oplus B\phi$ is an associative algebra by Definition {\rm \ref{df: morphism algebra}}. Then $(\phi !,d_{\phi !})$ is a differentail algebra.
\end{lem}

\begin{proof}
For $a+b_1+b_2\phi,a'+b'_1+b'_2\phi \in \phi !$, we have
\begin{align*}
&(a+b_1+b_2\phi)d_{\phi !}(a'+b'_1+b'_2\phi)+d_{\phi !}(a+b_1+b_2\phi)(a'+b'_1+b'_2\phi)\\
&+\lambda d_{\phi !}(a+b_1+b_2\phi)d_{\phi !}(a'+b'_1+b'_2\phi)\\
=&(a+b_1+b_2\phi)[d_A(a')+d_B(b'_1)+d_B(b'_2)\phi]+[d_A(a)+d_B(b_1)+d_B(b_2)\phi](a'+b'_1+b'_2\phi)\\
&+\lambda (d_A(a)+d_B(b_1)+d_B(b_2)\phi)(d_A(a')+d_B(b'_1)+d_B(b'_2)\phi)\\
=&ad_A(a')+b_1d_B(b'_1)+[b_2\phi(d_A(a'))+b_1d_B(b'_2)]\phi+d_A(a)a'+d_B(b_1)b'_1\\
&+[d_B(b_2)\phi(a')+d_B(b_1)b'_2]\phi+\lambda [d_A(a)d_A(a')+d_B(b)d_B(b'_1)+(d_B(b_2)\phi(d_A(a'))+d_B(b_1)d_B(b'_2))\phi]\\
=&[ad_A(a')+d_A(a)a'+\lambda d_A(a)d_A(a')]+[b_1d_B(b'_1)+d_B(b_1)b'_1+\lambda d_B(b_1)d_B(b'_1)]\\
&+[b_2d_B(\phi(a'))+d_B(b_2)\phi(a')+\lambda d_B(b_2)d_B(\phi(a'))]\phi\\
&+[b_1d_B(\phi(b'_2))+d_B(b_1)b'_2+\lambda d_B(b_1)d_B(b'_2)]\phi\\
=&d_A(aa')+d_B(b_1b'_1)+d_B(b_2\phi(a')+b_1b'_2)\phi\\
=&d_{\phi !}((a+b_1+b_2\phi)(a'+b'_1+b'_2\phi))
\end{align*}
So, $d_{\phi !}$ is a differential  operator of weight $\lambda$. It follow that $(\phi !,d_{\phi !})$ is a differential algebra of weight $\lambda$.
\end{proof}

Let $\phi:(A, d_{A})\rightarrow (B, d_{B})$ be a morphism of differential algebras, $\langle M,N ,\psi\rangle$ be a differential $\phi$-bimodule. We define $d_{\psi !}:\psi !\rightarrow \psi ! $ by $d_{\psi !}(m+n_1+n_2\phi)=d_M(m)+d_N(n_1)+d_N(n_2)\phi$, in the same way, one may check that $(\psi !,d_{\psi !})$ is a differential bimodule over $(\phi !,d_{\phi !})$. By Remark \ref{remark morphism module}, $\langle{}_{\rhd}M_\lhd,{}_{\rhd}N_\lhd ,_\triangleright\psi_\triangleleft\rangle$ is a differential $\phi$-bimodule. So $(_{\rhd}\psi_\lhd !, d_{_{\rhd}\psi_\lhd !})$ is a bimodule over differential algebra $(\phi !,d_{\phi !} )$.

\begin{df}{\rm (\cite{ms2})}
Let $\phi:A\rightarrow B$ be an associative algebra morphism, $\langle M,N ,\psi\rangle$ be a bimodule of $\phi$, define
$\tau^\bullet_\phi:\C^\bullet(\phi,\psi)\rightarrow \C^\bullet_{\mathrm{Alg}}(\phi!,\psi!)$ as follows:
for $f=(f^A,f^B,f^{AB})\in \C^n(\phi,\psi)$, $\tau^n_\phi f$ defined by
\begin{align*}
&\tau^n_\phi f|_{B^{\t n}}=f^B; \tau f|_{A^{\t n}}=f^A\\
&\textmd{for}~ ~(b\phi,a_2,\cdots,a_n)\in B\phi\t A^{n-1}\\
&\tau^n_\phi f(b\phi,a_2,\cdots,a_n)=f^B(b,\phi(a_2),\cdots,\phi(a_n))\phi +b f^{AB}(a_2,\cdots,a_n)\phi \\
&\textmd{for}~ ~(b_1,\cdots,b_{r-1},b_r\phi,a_{r+1},\cdots,a_n)\in B^{r-1}\t B\phi\t A^{n-r}\\
&\tau^n_\phi f(b_1,\cdots,b_{r-1},b_r\phi,a_{r+1},\cdots,a_n)=
f^B(b_1,\cdots,b_{r-1},b_r,\phi(a_{r+1}),\cdots,\phi(a_n))\phi   \\
&\tau^n_\phi f(x_1,\cdots,x_n)=0. \indent\indent otherwise
\end{align*}
Then, $\tau^\bullet_\phi:\C^\bullet(\phi,\psi)\rightarrow \C^\bullet_{\mathrm{Alg}}(\phi!,\psi!)$ is a quasi-isomorphism.
\end{df}

Let $\phi:(A, d_{A})\rightarrow (B, d_{B})$ be a differential algebra morphism of weight $\lambda$ and $\langle M,N ,\psi\rangle$ be a differential $\phi$-bimodule, then  $\langle{}_{\rhd}M_\lhd,{}_{\rhd}N_\lhd ,_\triangleright\psi_\triangleleft\rangle$ is a differential $\phi$-bimodule. Let $\C^\bullet_{\mathrm{Alg}}(\phi !,_\triangleright\psi_\triangleleft!)$ = $\C^\bullet_{\mathrm{DO_\lambda}}(\phi!,\psi!)$.
We will show the following lemma:

\begin{lem}\label{lem: commu diagram}
The diagram
\[
\xymatrix{
\C^\bullet(\phi,\psi) \ar[d]^{\pi^\bullet} \ar[r]^{\tau^\bullet_\phi} &\C^\bullet_{\mathrm{Alg}}(\phi!,\psi!)\ar[d]^{\Phi^\bullet}\\
\C^\bullet(\phi,_\triangleright\psi_\triangleleft) \ar[r]^{\tau^\bullet_{\phi_\triangleleft}} &\C^\bullet_{\mathrm{DO\lambda}}(\phi!,\psi!)}.
\] is commutative, i.e., $\Phi^\bullet\circ\tau^\bullet_\phi=\tau^\bullet_{\phi_\triangleleft}\circ\pi^\bullet$.
\end{lem}

This proof is straightforward but tedious calculation, so we omitt it.
Then, we define $\tau^\bullet:\cm^\bullet(\phi,\psi)\rightarrow \C^\bullet_{\mathrm{DA_\lambda}}(\phi!,\psi!)$ to be
$\tau^\bullet=\begin{pmatrix} \tau^{\bullet}_{\phi}&0\\0&\tau^{\bullet}_{\phi_\triangleleft}
\end{pmatrix}$. From  Lemma \ref{lem: commu diagram}, $\tau^\bullet$ is a chain map.
Thus, we get the following commutative diagram



\[\xymatrix{
		0\ar[r]& \C^{\bullet-1}(\phi,_\triangleright\psi_\triangleleft)\ar[r]\ar[d]^-{\tau^{\bullet-1}_{\phi_{\star}}}&\cm^\bullet(\phi,\psi)
\ar[r]\ar[d]^-{\tau^{\bullet}}&\C^\bullet(\phi,\psi)\ar[r]\ar[d]^-{\tau^{\bullet}_{\phi}}&0\\
		0\ar[r]&\C^{\bullet-1}_{\mathrm{DO_\lambda}}(\phi!,\psi!)\ar[r]&   \C^\bullet_{\mathrm{DA_\lambda}}(\phi!,\psi!)\ar[r]&  \C^\bullet_{\mathrm{Alg}}(\phi!,\psi!)\ar[r]&0. } \]

Since $\mathrm{H}^\bullet(\tau^{\bullet}_{\phi_{\triangleleft}})$ and $\mathrm{H}^\bullet(\tau^{\bullet}_{\phi})$ are isomorphisms, $\mathrm{H}^\bullet(\tau^\bullet)$ is also an isomorphism. It follows that the main theorem

\begin{thm}\label{thm: CCT of RB}
Suppose that $\phi:(A, d_{A})\rightarrow (B, d_{B})$ is a morphism of differential algebras of weight $\lambda$ and $\langle M,N ,\psi\rangle$ is a differential $\phi$-bimodule. Let  $\mathrm{H}_{\mathrm{mor}_\lambda}^n(\phi,\psi)$ be the cohomology group of $\phi$ with coefficients in $\langle M,N ,\psi\rangle$ and $\mathrm{H}_{\mathrm{DA}_\lambda}^n(\phi!,\psi!)$ be the cohomology gruop of differential algebra   $(\phi!,T_{\phi!})$ with coefficients in $(\psi!,T_{\psi!})$. Then $\mathrm{H}_{\mathrm{mor}_\lambda}^n(\phi,\psi)\cong\mathrm{H}_{\mathrm{DA}_\lambda}^n(\phi!,\psi!)$.
\end{thm}

\section{The case of $\lambda =0$}
When $\lambda =0$, operad of differential algebras is Koszul \cite{LO}, so we can construct minimal models of differential algebra morphisms operads by \cite{Do}. Recall that a dg operad is called quasi-free if its underlying graded operad is free. We will construct $2$-colored operads over differential algebras and their minimal models. Following \cite{CGWZ}, we recall some notions.
\begin{definition}\cite{CGWZ}
The operads for differential algebras is defined to be the quotient of
the free graded operad $F(M)$ generated by a graded collection $M$ by an operadic ideal $I$, where
the collection $M$ is given by $M(1) = kd, M(2) = k\mu$ and $M(n) = 0$ for $n\neq 1, 2$ and where $I$ is
generated by
\begin{align*}
&\mu\circ_1 \mu-\mu\circ_2\mu, \\
&d \circ_1\mu-(\mu\circ_1 d+\mu\circ_2 d+\lambda(\mu\circ_1d)\circ_2d).
\end{align*}
Denote this operad by $_\lambda\mathfrak{Dif}$.
\end{definition}
Then minimal models of  $_\lambda\mathfrak{Dif}$ may be defined in \cite{CGWZ} as follows.
\begin{definition}\label{def-minimodel1}\cite{CGWZ}
Let $\mathcal{O} = (\mathcal{O}(1),..., \mathcal{O}(n),...)$ be the graded collection where $\mathcal{O}(1) = kd_1$ with
$|d_1|=0$ and for $n \geqslant 2$, $\mathcal{O}(n) = kd_n\oplus k m_n$ with $|d_n| = n-1, |m_n|= n-2$. The operad $_\lambda\mathfrak{Dif}_{\infty}$
of homotopy differential algebras is defined by the differential graded operad $(\mathcal{F(O)}, \partial)$, where
the underlying free graded operad is generated by the graded collection $\mathcal{O}$ and the action of the
diﬀerential $\partial$ on generators is given by the following equations. For each $n\geqslant 2$,
\begin{align}\label{eq-partial1}
\partial(m_n)=\sum^{n-1}_{j=2} \sum^{n-j+1}_{i=1} (-1)^{i+j(n-i)} m_{n-j+1}\circ_i m_j
\end{align}
and for $n\geqslant 1$,
\begin{align}\label{eq-partial2}
\partial(d_n)&=-\sum_{j=2}^{n}\sum_{i=1}^{n-j+1}(-1)^{i+j(n-i)}d_{n-j+1}\circ_i m_j \nonumber\\
&-\sum_{{{1\leqslant k_1<...<k_q\leqslant p;}\atop{l_1+...+l_q+p-q=n;}}\atop{l_1,...,l_q\geqslant1,2\leqslant p\leqslant n,1\leqslant q \leqslant p}}
(-1)^\xi\lambda^{q-1}(...(((m_p\circ_{k_1}d_{l_1})\circ_{k_2+l_1-1}d_{l_2})\nonumber\\
&\circ_{k_3+l_1+l_2-2}d_{l_3}...)\circ_{k_q+l_1+...+l_{q-1}}d_{l_q}
\end{align}
where $\xi:=\sum^q_{s=1}(l_s-1)(p-k_s)$. And  $_\lambda\mathfrak{Dif}_{\infty}$ is the minimal model of the operad $_\lambda\mathfrak{Dif}$ of differential algebras.
\end{definition}

Similar to \cite[Example 4]{Mark4}, we have the following definition.
\begin{definition}
Let $_\lambda\mathfrak{Dif}$ be the operad for differential algebras. Then there is a $2$-colored operad $_\lambda\mathfrak{Dif}^{\mathbb{A}\rightarrow \mathbb{B}}$  whose algebras are of the
form $f : A \rightarrow B$ , in which $A$ and $B$ are $_\lambda\mathfrak{Dif}$-algebras and $f$ is a morphism of $_\lambda\mathfrak{Dif}$-algebras, which is constructed as the following
\begin{align*}
_\lambda\mathfrak{Dif}^{\mathbb{A}\rightarrow \mathbb{B}}=\frac{_\lambda\mathfrak{Dif}^A*_\lambda\mathfrak{Dif}^B*F(f)}{(fx_A=x_Bf^\otimes,x\in _\lambda\mathfrak{Dif}(n))},
\end{align*}
where $_\lambda\mathfrak{Dif}^A$ and $_\lambda\mathfrak{Dif}^B$ are copies of $_\lambda\mathfrak{Dif}$  concentrated in the colors $\mathbb{A}$ and $\mathbb{B}$, respectively,
$x_A$ and $x_B$ are the respective copies of $x$ in $_\lambda\mathfrak{Dif}^A$ and $_\lambda\mathfrak{Dif}^B$, and $F(f)$ is the
free $2$-colored operad on the generator $f : A\rightarrow B$. The star $*$ denotes the free product (= the coproduct) of $2$-colored operad.

Precisely, $_\lambda\mathfrak{Dif}^{\mathbb{A}\rightarrow \mathbb{B}}$ is defined to be the quotient of
the free graded operad $F(N)$ generated by a graded collection $N$ by an operadic ideal $I$, where
the collection $N$ is given by $M(1) = k\{d_A,d_B,f\}, M(2) = k\{\mu_A,\mu_B\}$ and $M(n) = 0$ for $n\neq 1, 2$ and where $I$ is
generated by
\begin{align*}
&\mu_{X}\circ_1 \mu_{X}-\mu_{X}\circ_2\mu_{X}, \\
&d_X \circ_1\mu_X-(\mu_X\circ_1 d_X+\mu_X\circ_2 d_X+\lambda(\mu_X\circ_1d_X)\circ_2d_X),\\
&f\circ_1\mu_A-(\mu_B\circ_1 f)\circ_2f,\\
&f\circ_1 d_A-d_B\circ_1 f,
\end{align*}
where $X=\{A,B\}$.
\end{definition}

When $\lambda =0$, the operad $_\lambda\mathfrak{Dif}$ is Koszul. So we can construct the minimal model of $_\lambda\mathfrak{Dif}^{\mathbb{A}\rightarrow \mathbb{B}}$ by Dotsenko-Poncin \cite{Do}. Recall homotopy cooperad $_\lambda\mathfrak{Dif}^{\text{!`}}$ in \cite{CGWZ} as follows.

The homotopy cooperad $\mathscr{S}(_\lambda\mathfrak{Dif}^{\text{!`}})$ is defined on the graded collection with arity-n component
$$\mathscr{S}(_\lambda\mathfrak{Dif}^{\text{!`}})=k \widetilde{m}_n+k \widetilde{d}_n$$
with $|\widetilde{m}_n|=0, |\widetilde{d}_n|=1$ for $n\geq 1$. Then the coaugmented homotopy cooperad structure on the graded collection $\mathscr{S}(_\lambda\mathfrak{Dif}^{\text{!`}})=k \widetilde{m}_n+k \widetilde{d}_n$ is defined as \cite[Section 3.2]{CGWZ}.

When $\lambda= 0$, we denote $_\lambda\mathfrak{Dif}^{\mathbb{A}\rightarrow \mathbb{B}}$ by $\mathfrak{Dif}_{\bullet\to\bullet}$ and $\mathscr{S}(_\lambda\mathfrak{Dif}^{\text{!`}})\triangleq \mathfrak{Dif}^{\text{!`}}$.
Let us consider $\{A,B\}$-colored $\mathbb{S}$-module 
$$\mathcal{M}_{\bullet\to \bullet}=\mathfrak{\overline{Dif}}^{\text{!`}}_{A\to A}\oplus \mathfrak{\overline{Dif}}^{\text{!`}}_{B\to B}\oplus s \mathfrak{Dif}^{\text{!`}}_{A\to B}.$$

We define the cobar complex $\Omega(\mathcal{M}_{\bullet\to \bullet})$ by $\mathfrak{Dif}_{\bullet\to \bullet,\infty}.$ By \cite{Do},it follows that 

\begin{prop}
$\mathfrak{Dif}_{\bullet\to \bullet,\infty}$ is the minimal model of $\mathfrak{Dif}_{\bullet\to\bullet}$.
\end{prop}
The general construction  produces an $L_\infty$-algebra structure on the space of $\mathbb{S}$-module morphisms
$$L_{A,B}=\Hom (\mathcal{M}_{\bullet\to \bullet}, End_{A\oplus B}),$$
where $End_{A\oplus B}$ is $\{A,B\}$-colored operad. This space of morphisms can be naturally identified with the space
$$(f_a,f_b, f_{ab})\in \Hom_{K} (\overline{\mathfrak{Dif}}^{\text{!`}}(A), A)\oplus\Hom_{K} (\overline{\mathfrak{Dif}}^{\text{!`}}(B), B)\oplus\Hom_{K} (s\mathfrak{Dif}^{\text{!`}}(A), B).$$  By means of \cite{Me} for properads, $(f_a,f_b, f_{ab})$ is a solution to the Maurer–Cartan equation of the $L_\infty$-algebra $L_{A,B}$ if and only if $f_a$ is a structure of a homotopy $\mathfrak{Dif}$-algebra on $A$,  $f_b$ is a structure of a homotopy $\mathfrak{Dif}$-algebra on B , and $f_{ab}$ is a homotopy morphism between these algebras.

\begin{remark}
A interesting question: for a differential algebra with $\lambda\neq 0$, how to construct its minimal model? Since it is not Koszul, the question is still hard for us.
\end{remark}

\hspace*{1cm}\\


\begin{thebibliography}{00}
\bibitem{BA} D. Balavoine. Deformations of algebras over a quadratic operad, {\it Contemp. Math.}  202 (1997) 207-234.
\bibitem{CGWZ} J. Chen, L. Guo, K. Wang, G. Zhou. Koszul duality, minimal model and $L_\infty$-structure for
differential algebras with weight. Adv. Math., 437 (2024) Paper No. 109438.


\bibitem{Do}V. Dotsenko and N. Poncin. A tale of three homotopies. Appl. Categ. Structures, 24(6) (2016) 845–873.


\bibitem{FMY} Y. Fr\'{e}gier, M. Markl, D. Yau The $L_\infty$-deformation complex of diagrams of algebras. {\it New York J. Math.} 15 (2009), 353–392.
\bibitem{GE}I. M. Gelfand, I. Ya. Dorfman, Hamiltonian operators and algebraic structures related to them, Funct. Anal.
Appl. 13 (1979) 248-262.
\bibitem{GS1}M. Gerstenhaber. On the deformation of rings and algebras. {\it Ann. of Math.} 79 (1964) 59-103.
\bibitem{ms1} M. Gerstenhaber, S. D. Schack, On the deformation of algebra morphisms and diagrams, {\it Trans. Amer. Math. Soc.} 279 (1983), no. 1, 1–50.
\bibitem{ms2} M. Gerstenhaber and S.D. Schack, On the cohomology of an algebra morphism, J. Algebra 95 (1985), 245-262.
\bibitem{GF}L. Guo, F. Li, Structure of Hochschild cohomology of path algebras and diﬀerential formulation of Euler’s
polyhedron formula, Asian J. Math. 18 (2014) 545-572.
\bibitem{GY} L. Guo, Y. Li, Y. Sheng, G. Zhou, Cohomologies, extensions and deformations of diﬀerential algebras with
arbitrary weight, {\it Theory Appl. Categ.} 8 (2022) 1409-1433.
\bibitem{KA} I. Kaplansky, An Introduction to Diﬀerential Algebra, Hermann, 1976.
\bibitem{KO}E. R. Kolchin, Diﬀerential Algebras and Algebraic Groups, Academic Press, New York, 1973.
\bibitem{LO}J.-L. Loday, On the operad of associative algebras with derivation, Georgian Math. J. 17 (2010) 347-372.
\bibitem{MA}A. R. Magid, Lectures on Diﬀerential Galois Theory, Amer. Math. Soc. 1994.
\bibitem{Mark1} M. Markl. Cotangent cohomology of a category and deformations. J. Pure Appl. Alg. 113 (1996) 195–218.
\bibitem{Mark2} M. Markl. Models for operads. Comm. Alg. 24 (1996) 1471–1500.
\bibitem{Mark3} M. Markl. Homotopy diagrams of algebras. Rend. Circ. Mat. Palermo (2) Suppl. 69 (2002) 169–180.
\bibitem{Mark4} M. Markl. Homotopy algebras are homotopy algebras. Forum Math. 16 (2004) 129–160.
\bibitem{Me} S. Merkulov,  B.Vallette. Deformation theory of representations of prop(erad)s. I. J. Reine Angew. Math. 634 (2009) 51–106. 
\bibitem{lie} A. Nijenhuis, R. Richardson. Cohomology and Deformations in Graded Lie
			Algebras, {\it Bull. Amer. Math. Soc.}  72 (1966) 1-29.
\bibitem{Po} L. Poinsot, Diﬀerential (Lie) algebras from a functorial point of view, Adv. Appl. Math. 72 (2016) 38-76.
\bibitem{Put} M. Van der Put, M. Singer, Galois Theory of Linear Diﬀerential Equations, Springer, 2003.
\bibitem{Ritt} J. F. Ritt.  Diﬀerential equations from the algebraic standpoint, Amer. Math. Soc., 1934.

\bibitem{MS1}
M. Gerstenhaber, S. D. Schack, On the deformation of algebra morphisms and diagrams, {\it Trans. Amer. Math. Soc.} 279 (1983), no. 1, 1–50.


\end{thebibliography}
\end{document}